\documentclass[11pt]{article}

\usepackage{layout}
\usepackage[makeroom]{cancel}
\usepackage{bbm}
\usepackage[affil-it]{authblk}

\usepackage{amsfonts}
\usepackage{amsmath}
\usepackage{amssymb}
\usepackage{amsthm}
\usepackage{graphicx} \usepackage{enumerate} \usepackage{multicol}
\usepackage{mathrsfs} \usepackage[all,cmtip]{xy}
\usepackage{enumerate}
\usepackage{xcolor}

\newtheorem{theorem}{Theorem}[section]

\newtheorem{proposition}[theorem]{Proposition}

\newtheorem{corollary}[theorem]{Corollary}
\newtheorem{lemma}[theorem]{Lemma}
\newtheorem{remark}[theorem]{Remark}
\newtheorem{example}[theorem]{Example}
\newtheorem{examples}[theorem]{Examples}
\newtheorem{foo}[theorem]{Remarks}

\newcommand{\calF}{\mathcal{F}}

\newcommand{\calL}{\mathcal{L}}
\newcommand{\M}{\mathbb{M}}
\newcommand{\bM}{\mathbb{M}}

\newcommand{\calV}{\mathcal{V}}
\newcommand{\Dh}{\Delta_\mathcal{H}}
\newcommand{\Dhe}{\Delta_{\mathcal{H},\varepsilon}}

\newcommand{\V}{\mathcal{V}}
\newcommand{\ch}{\mathcal{H}}
\newcommand{\ee}{\ell}

\DeclareMathOperator{\End}{End}
\DeclareMathOperator{\id}{id}

\DeclareMathOperator{\tr}{tr}

\newcommand{\eps}{\varepsilon}
\newcommand{\ve}{\varepsilon}

\DeclareMathOperator{\Ric}{Ric}

\title{Transverse Weitzenb\"ock formulas and de Rham cohomology of totally geodesic foliations}

\author{Fabrice Baudoin%
  \thanks{Author supported in part by the NSF Grant DMS 1660031}}
\affil{Department of Mathematics, University of Connecticut, USA\par
 \texttt{fabrice.baudoin@uconn.edu}}

\author{Erlend Grong%
\thanks{Author supported in part by the Research Council of Norway (project number 249980/F20).}}
\affil{Universit\'e Paris-Sud, LSS-SUP\'ELEC, \par
3, rue Joliot-Curie, 91192 Gif-sur-Yvette, France, and \par
   University of Bergen, Department of Mathematics, \par P.O. Box 7803, 5020 Bergen, Norway\par
   \texttt{erlend.grong@math.uib.no}}

\date{}

\begin{document}

\maketitle

\begin{abstract}
We prove transverse Weitzenb\"ock  identities for the horizontal Laplacians of a totally geodesic foliation. As a   consequence, we  obtain nullity theorems for the de Rham cohomology assuming only the positivity of curvature quantities  transverse to the leaves. Those curvature quantities appear in the adiabatic limit of the canonical variation of the metric.
 \end{abstract}

%\tableofcontents

%\newpage

\section{Introduction}

Celebrated Bochner's technique \cite{B,Yano} yields that a compact oriented Riemannian manifold $(\M,g)$ with positive Ricci curvature has a first de Rham cohomology group $H_{dR}^1(\M)=0$. This beautiful result, connecting a curvature quantity that can be computed locally to the topology of the manifold, has been at the origin of some of the most exciting developments of the 20th century in the global analysis of manifolds (like Hodge and index theory). Moreover, if the curvature operator itself is positive, the Bochner technique implies that $H^k_{dR}(\M) = 0$ for $k \neq 0, \dim \M$, see \cite{Mey71,GaMe73}. In fact, using Ricci flow, it was shown in \cite{BoWi08} that that any simply connected Riemannian manifold with positive curvature operator is diffeomorphic to a sphere. In the present paper, we study analogues of Bochner's technique in the context of a totally geodesic Riemannian foliation $(\M,g,\mathcal{F})$ whose horizontal distribution $\mathcal{F}^\perp$ is bracket generating. 

The study of cohomological properties of Riemannian foliations by the Bochner's technique is not new and the literature is extensive (see for instance \cite{AT,Habib,KT,MRT,Park} and the bibliography in \cite{TondeurBook1988}). However, in all of those works the authors are interested in the so-called basic cohomology, which can be understood as the cohomology of the leaf space. In our work, we are rather  interested in the \textit{full de Rham cohomology} of the manifold. It should be noted that obtaining informations about the de Rham cohomology from horizontal curvature quantities only,  requires the bracket-generating condition for the horizontal distribution. Unlike \cite{MRT,Park}  the Laplacians we consider will therefore  be hypoelliptic but not elliptic.

We work in a general framework of totally geodesic Riemannian foliations, but we emphasize that our results are essentially new even in the context of contact manifolds. In fact, to provide the reader with an initial motivation for our more general results, we begin with stating our main objective in such setting. 

\begin{theorem}
Let $(\M,g)$ be a compact, connected, oriented K-contact manifold with positive horizontal Tanno Ricci curvature, then $H_{dR}^1(\M)=0$. If moreover $\M$ has a positive horizontal Tanno curvature operator then $H_{dR}^k(\M)=0$ for $0 < k < \dim \M$.
\end{theorem}

Due to the differences that will appear between the one-form case and the $k$-form case, $k \ge 2$, we treat separately those two cases. The paper is organized as follows.  In Section 2, we focus first on transverse Weitzenb\"ock formulas on one-forms. We prove that the horizontal Laplacian on one-forms that was first introduced in \cite{BKW} can actually  be written as $\Dh=-\delta_{\mathcal{H}} d -d\delta_{\mathcal{H}} $ where $\delta_{\mathcal{H}}$ is a horizontal divergence computed with respect to some connnection. Taking then advantage of the Weitzenb\"ock formula proved in \cite{BKW}, this allows us to apply Bochner's technique and obtain a sufficient criterion for the nullity of the first de Rham cohomology group. There are two difficulties to overcome. The first one is that the operator $\Dh$ is only hypoelliptic but not elliptic and the second one that, in general, it is not even symmetric. In Section 3, we propose a very general construction for horizontal Laplacians in the framework of sub-Riemannian geometry, generalizing to arbitrary forms the results of of \cite{GrTh16}. Weitzenb\"ock formulas for those sub-Laplacians are then studied. In Section 4, we use the horizontal Laplacians constrcucted in Section 3 to perform the horizontal Bochner's method on $k$-forms, $k \ge 2$.

\

To conclude the introduction, we mention the monograph \cite{EJL99}. Theorem 2.4.2 in \cite{EJL99} states a quite general Weitzenb\"ock type formula for general H\"ormander's type operators on forms  that can be compared to our transverse Weitzenb\"ock formula Proposition \ref{prop:SLForms}. Then Corollary 3.3.14  in  \cite{EJL99} states then a corresponding vanishing theorem for the de Rham cohomology under the positivity assumption of the Weitzenb\"ock term. However,  the authors work with general H\"ormander's type operators on forms, without discussing the notion of intrinsic canonical horizontal Laplacians.   In comparison, in the present work, by using the canonical variation of the metric, we construct geometrically meaningful horizontal Laplacians. We show then that the  curvature tensor $Q$  controlling (in the Bochner's sense) the de Rham cohomology appear in the adiabatic limit of the Weitzenb\"ock terms of those horizontal Laplacians (see Remark \ref{optimality} for the details about the optimality of $Q$). In the case where the foliation is Yang-Mills, this tensor $Q$ can always simply be reduced to the Ricci-Weitzeb\"ock curvature of the leaf space $\M / \mathcal{F}$.

\section{Horizontal Bochner's method on one-forms}

Our goal in this section is to prove the following generalization of Hodge theorem:

\begin{theorem}\label{H1 1forms}
 Let $\M$ be a smooth, connected, oriented and compact  manifold . We assume that $\bM$ is equipped with a  Riemannian foliation $\mathcal{F}$ with bundle-like  metric $g$ and totally geodesic   leaves. We also assume that the tensor $Q$ defined in  \eqref{def R} is positive. Then, $H_{dR}^1(\M)=0$.
\end{theorem}

The various assumptions of the theorem are explained in this section. We mention that in the contact case, the theorem yields:

\begin{corollary}
 Let $\M$ be a smooth, connected, oriented and compact K-contact manifold with positive horizontal Tanno Ricci curvature, then $H_{dR}^1(\M)=0$.
\end{corollary}

\subsection{Transverse Weitzenb\"ock formulas on one-forms} \label{sec:TransverseWeitz}

We first recall the framework and notations of \cite{BaudoinEMS2014,BKW}, to which we refer for further details. Let~$\M$ be a smooth, oriented, connected, compact manifold with dimension $n+m$. We assume that~$\bM$ is equipped with a Riemannian foliation $\mathcal{F}$ with a bundle-like complete metric $g$ (see \cite{Reinhart1959a}) and totally geodesic  $m$-dimensional leaves. The distribution $\mathcal{V}$ formed by vectors tangent to the leaves is referred  to as the set of \emph{vertical directions}. The distribution $\mathcal{H}$ which is normal to $\mathcal{V}$ is referred to as the set of \emph{horizontal directions}.  We will always assume in this paper that the horizontal distribution $\mathcal{H}$ is everywhere step two bracket-generating.

\begin{example}
Let $(\M,\theta)$ be a $2n+1$-dimensional smooth contact manifold. There is a unique smooth vector field $Z$, the so-called Reeb vector field, that satisfies
\[
\theta(Z)=1,\qquad \mathcal{L}_Z\theta=0,
\]
where $\mathcal{L}_Z$ denotes the Lie derivative with respect to  $Z$. On $\M$ there is a foliation, the Reeb foliation, whose leaves are the orbits of the vector field $Z$.  As it is well-known (see  \cite{Tanno}), it is always possible to find a Riemannian metric $g$ and a $(1,1)$-tensor field $J$ on $\M$ so that for every  vector fields $X, Y$
\[
g(X,Z)=\theta(X),\quad J^2(X)=-X+\theta (X) Z, \quad g(JX,Y)= d\theta(X,Y).
\]
 The triple $(\M, \theta,g)$ is called a contact Riemannian manifold. It is well-known (\cite{Tanno}) that the Reeb foliation is totally geodesic with bundle like metric if and only if the Reeb vector field $Z$ is a Killing field, that is,
\[
\mathcal{L}_Z g=0.
\]
In that case, $(\M, \theta,g)$ is called a K-contact Riemannian manifold. Observe that the horizontal distribution $\mathcal{H}$ is then the kernel of $\theta$ and that $\mathcal{H}$ is bracket generating because $\theta$ is a contact form. Sasakian manifolds are examples of $K$-contact manifolds. We refer to  \cite{BoyerGalickiBook} for further details on Sasakian foliations.
\end{example}

The reference connection on $\M$ (see \cite{AT,BaudoinEMS2014,Hladky}), the Bott connection, is given as follows:

\[
\nabla_X Y =
\begin{cases}
\pi_{\mathcal{H}} ( \nabla_X^g Y) , X,Y \in \Gamma^\infty(\mathcal{H}) \\
\pi_{\mathcal{H}} ( [X,Y]),  X \in \Gamma^\infty(\mathcal{V}), Y \in \Gamma^\infty(\mathcal{H}) \\
\pi_{\mathcal{V}} ( [X,Y]),  X \in \Gamma^\infty(\mathcal{H}), Y \in \Gamma^\infty(\mathcal{V}) \\
\pi_{\mathcal{V}} ( \nabla_X^g Y) , X,Y \in \Gamma^\infty(\mathcal{V})
\end{cases}
\]
where $\nabla^g$ is the Levi-Civita connection and $\pi_\mathcal{H}$ (resp. $\pi_\mathcal{V}$) the projection on $\mathcal{H}$ (resp. $\mathcal{V}$). It is easy to check that  this connection satisfies $\nabla g=0$. The Bott connection has a torsion which is given by:
\[
T(X,Y)=-\pi_\mathcal{V} ([X,Y]).
\]

\begin{example}
Let $(\M, \theta,g)$ be a K-contact  manifold. The Bott connection coincides  with the Tanno's connection that was introduced in \cite{Tanno} and which is the unique connection that satisfies:
\begin{enumerate}
\item $\nabla\theta=0$;
\item $\nabla Z=0$;
\item $\nabla g=0$;
\item ${T}(X,Y)=d\theta(X,Y)Z$ for any $X,Y\in \Gamma^\infty(\mathcal{H})$;
\item ${T}(Z,X)=0$ for any vector field $X\in \Gamma^\infty(\mathcal{H})$.
\end{enumerate}
In the case where $(\M, \theta,g)$ is Sasakian, the Bott connection coincides then with the Tanaka-Webster connection (see \cite{dragomir}). 
\end{example}

We define the horizontal gradient $\nabla_\mathcal{H} f$ of a smooth function $f$ as the projection of the Riemannian gradient of $f$ on the horizontal bundle. Similarly, we define the vertical gradient $\nabla_\mathcal{V} f$ of a function $f$ as the projection of the Riemannian gradient of $f$ on the vertical bundle.
The horizontal Laplacian $\Dh$ is the generator of the symmetric pre-Dirichlet form
\[
\mathcal{E}_{\mathcal{H}} (f,g) =\int_\bM \langle \nabla_\mathcal{H} f , \nabla_\mathcal{H} g \rangle_{\mathcal{H}} d\mu, \quad f,g \in C_0^\infty(\M),
\]
where $\mu$ is the Riemannian volume measure.  We have therefore the following integration by parts formula
\[
\int_\bM \langle \nabla_\mathcal{H} f , \nabla_\mathcal{H} g \rangle_{\mathcal{H}} d\mu=-\int_\M f \Dh g d\mu=-\int_\M g \Dh f d\mu, \quad f,g \in C_0^\infty(\M).
\]
From this convention $\Dh$ is therefore non-positive. The Riemannian metric $g$ can be split as
\[
g=g_\mathcal{H} \oplus g_{\mathcal{V}},
\]
and for later use, we introduce its canonical variation, which is the the one-parameter family of Riemannian metrics defined by:
\[
g_{\varepsilon}=g_\mathcal{H} \oplus  \frac{1}{\varepsilon }g_{\mathcal{V}}, \quad \varepsilon >0.
\]
The limit $\ve \to 0$ is the \textit{sub-Riemannian} limit and the limit $\ve \to +\infty$ the so-called \textit{adiabatic} limit.
\begin{remark}
The canonical variation of the metric of a totally geodesic foliation has been studied for a long time (see Chapter 9 in \cite{Besse} and the references therein). A prototype example is given by the canonical variation  of the standard metric $g$ on the odd-dimensional sphere $\mathbb{S}^{2n+1}$ foliated by the Hopf fibration. The Riemannian manifold $(\mathbb{S}^{2n+1}, g_{\ve})$ is  then called the Berger\footnote{For Marcel Berger (1927-2016).} sphere. When $\ve \to 0$, in the Gromov-Hausdorff sense,  $(\mathbb{S}^{2n+1}, g_{\ve})$ converges to $\mathbb{S}^{2n+1}$ endowed with the Carnot-Carath\'eodory metric. When $\ve \to \infty$, $(\mathbb{S}^{2n+1}, g_{\ve})$ converges to the complex projective space $\mathbb{CP}^n$ endowed with the standard Fubini-Study metric.
\end{remark}

For $Z \in \Gamma^\infty(\V)$, there is a  unique skew-symmetric endomorphism  $J_Z:\mathcal{H}_x \to \mathcal{H}_x$ such that for all horizontal vector fields $X$ and $Y$,
\begin{align}\label{Jmap}
g_\mathcal{H} (J_Z X,Y)= g_\mathcal{V} (Z,T(X,Y)).
\end{align}
where $T$ is the torsion tensor of $\nabla$. We then extend $J_{Z}$ to be 0 on  $\mathcal{V}_x$. Also, if $Z\in \Gamma^\infty (\ch)$, from \eqref{Jmap} we set $J_Z=0$.

 If $Z_1,\dots,Z_m$ is a local vertical frame, the operator $\sum_{\ee=1}^m J_{Z_\ee}J_{Z_\ee}$ does not depend on the choice of the frame and shall concisely be denoted by $\mathbf{J}^2$. 

\begin{example} If $\M$ is a K-contact manifold equipped with the Reeb foliation, then $\mathbf{J}$ is the almost complex structure on the horizontal bundle, and therefore $\mathbf{J}^2=-\mathbf{Id}_{\mathcal{H}}$.
\end{example}

\

 The horizontal divergence of the torsion $T$ is the $(1,1)$ tensor  which is defined in a local horizontal frame $X_1,\dots,X_n$ by
\[
\delta_\mathcal{H} T (X)= -\sum_{j=1}^n(\nabla_{X_j} T) (X_j,X).
\]
The $g$-adjoint of $\delta_\mathcal{H}T$ will be denoted $\delta^*_\mathcal{H} T$.
The foliation $(\M, \mathcal F, g)$ is said to be Yang-Mills if $\delta_\mathcal{H} T=0$ (see \cite{BaudoinEMS2014} or \cite{Besse}).

\begin{example} If $\M$ is a K-contact manifold, then the Reeb foliation is Yang-Mills (see \cite{Agrachev}).
\end{example}

 By declaring a one-form to be horizontal (resp. vertical) if it vanishes on the vertical bundle $\mathcal{V}$ (resp. on the horizontal bundle $\mathcal{H}$), the splitting of the tangent space
 \[
 T_x \bM= \mathcal{H}_x \oplus \mathcal{V}_x
 \]
 gives a splitting of the cotangent space. The metric $g_\varepsilon$ induces  then a metric on the cotangent bundle which we still denote $g_\varepsilon$. By using similar notations and conventions as before we define pointwisely for every $\eta$ in $T^*_x \M$,
\[
\| \eta \|^2_{\varepsilon} =\| \eta \|_\mathcal{H}^2+\varepsilon \| \eta \|_\mathcal{V}^2.
\]

\

By using the duality given by the metric $g$, $(1,1)$ tensors can also be seen as linear maps on the cotangent bundle $T^* \M$. More precisely, if $A: \Gamma^\infty(T\M)\to \Gamma^\infty(T\M)$ is a $(1,1)$ tensor, we will  often still denote by $A$ the fiberwise linear map on the cotangent bundle which is defined as the $g$-adjoint of the dual map of $A$. Namely $A:\Gamma^\infty(T^*\M)\to \Gamma^\infty(T^*\M)$ is such that for any $\eta, \xi\in \Gamma(T^*\M)$, $\langle A\eta,\xi\rangle=\xi(A\sharp \eta)$ where $\sharp$ is the standard musical isomorphism. The same convention will be made for any $(r,s)$ tensor. As a convention, unless explicitly mentioned otherwise in the text, the inner product duality will always be understood with respect to the reference metric $g$ (and not $g_\varepsilon$).

\

We define then the horizontal Ricci curvature $\Ric_{\mathcal{H}}$ of the Bott connection as the fiberwise symmetric linear map on one-forms such that for every smooth functions $f,g$,
\[
\langle  \Ric_{\mathcal{H}} (df), dg \rangle=\Ric (\nabla_\mathcal{H} f ,\nabla_\mathcal{H} g),
\]
where $\Ric$ is the Ricci curvature of the connection $\nabla$.

\begin{remark}
If the foliation comes from a submersion, then $\Ric_{\mathcal{H}}$ is simply the horizontal lift of the Ricci curvature of the leaf space.
\end{remark}

The adjoint connection of the Bott connection is not metric. For this reason, we shall rather make use of the following family of connections first introduced in \cite{Bau14,BKW}:
\[
\nabla^\varepsilon_X Y= \nabla_X Y -T(X,Y) +\frac{1}{\varepsilon} J_Y X, \quad 0 < \ve \le +\infty
\] 
and we shall only keep the Bott connection as a reference connection. 
It is readily checked that $\nabla^\varepsilon g_\varepsilon =0$ when $\ve <+\infty$. The adjoint connection (see Appendix A) of $\nabla^\varepsilon$ is then given by
\[
\hat{\nabla}^\varepsilon_X Y=\nabla_X Y +\frac{1}{\varepsilon} J_X Y,
\]
thus $\hat{\nabla}^\varepsilon$ is also a metric connection. It moreover preserves the horizontal and vertical bundle.  If $\eta$ is a one-form, we define the horizontal gradient in a local adapted frame of $\eta$ as the $(0,2)$ tensor
\[
\nabla_\mathcal{H} \eta =\sum_{i=1}^n \nabla_{X_i} \eta \otimes \theta_i.
\]
where $\theta_i, i=1,\dots, n$ is the dual of $X_i$. Finally, for $\varepsilon >0 $, we consider the following  operator which is defined on one-forms by
\begin{align}\label{def}
\Dhe=-(\nabla^\ve_\mathcal{H})^* \nabla_\mathcal{H}^\ve+\frac{1}{\varepsilon} \delta_\mathcal{H} T-\frac{1}{ \varepsilon}\mathbf{J}^2- \Ric_{\mathcal{H}},
\end{align}
where the adjoint  $^*$ is understood with respect to the $(L^2,g_\varepsilon)$ product on sections, i.e. $\int_\M \langle\cdot,\cdot\rangle_{\varepsilon}d\mu$. It is easily seen that, in a local horizontal  frame,
\begin{align}\label{eq-L-form}
-(\nabla^\ve_\mathcal{H})^* \nabla_\mathcal{H}^\ve
=\sum_{i=1}^n (\nabla_{X_i}^\ve)^2 -  \nabla^\ve_{\nabla^\ve_{X_i} X_i}.
\end{align}

\begin{remark}\label{YM symmetry}
We observe that $\Dhe$ is a  symmetric operator with respect to the $(L^2,g_\varepsilon)$ product on sections if and only if $\delta_\mathcal{H} T=0$, that is if and only if the foliation is of Yang-Mills type.
\end{remark}
It is proved in \cite{BKW} that for every smooth one-form $\alpha$ on $\M$, the following holds
\[
\lim_{\varepsilon \to \infty} \Dhe \alpha=\Delta_{\mathcal{H},\infty} \alpha,
\]
where, in a local horizontal frame
\[
\Delta_{\mathcal{H},\infty}=\sum_{i=1}^n (\nabla_{X_i}^\infty)^2 -  \nabla^\infty_{\nabla^\infty_{X_i} X_i}- \Ric_{\mathcal{H}}.
\]

Our first result is the following transverse Weitzenb\"ock formula on one-forms.

\begin{proposition}\label{W3}
Let $0<\varepsilon \le +\infty$. On one-forms, one has
\[
\Dhe=-d \delta_{\mathcal{H}, \ve}-\delta_{\mathcal{H}, \ve} d,
\]
where $\delta_{\mathcal{H}, \ve}$ is the horizontal divergence for the connection $\nabla^\ve$ defined for a $p$-form $\omega$  as the $p-1$ form:
\[
\delta_{\mathcal{H}, \ve} \omega (V_1,\cdots,V_{p-1})=-\sum_{i=1}^n \nabla_{X_i}^\ve \omega ( X_i ,V_1,\cdots,V_{p-1}),
\]
in a local orthonormal horizontal frame $X_1,\cdots,X_n$.
\end{proposition}

\begin{proof}
The proposition will be proved in  greater generality in Proposition \ref{prop:SLForms}. We give here a sketch of  direct proof that allows to identify explicitly the tensors on 1-forms and  comparison to \cite{BKW}. Let $x \in \M$. From \cite{BKW}, around $x$, there exist a local orthonormal horizontal frame $\{X_1,\cdots,X_n \}$ and a local orthonormal vertical frame $\{Z_1,\cdots,Z_m \}$ such that the following structure relations hold
\[
[X_i,X_j]=\sum_{k=1}^n \omega_{ij}^k X_k +\sum_{k=1}^m \gamma_{ij}^k Z_k
\]
\[
[X_i,Z_k]=\sum_{j=1}^m \beta_{ik}^j Z_j,
\]
where $\omega_{ij}^k,  \gamma_{ij}^k,  \beta_{ik}^j $ are smooth functions such that:
\[
 \beta_{ik}^j=- \beta_{ij}^k.
\]
Moreover, at $x$, we have
\[
 \omega_{ij}^k=0,  \beta_{ij}^k=0.
\]
We denote by $\theta_1,\cdots,\theta_n$ the dual basis of $X_1,\cdots,X_n$ and $\nu_1,\cdots,\nu_m$ the dual basis of $Z_1,\cdots,Z_m$.
We now use Fermion calculus on one-forms and introduce the following operators acting on one-forms:
\[
a_i^*\alpha=\theta_i \wedge \alpha, \quad a_i \alpha=\iota_{X_i} \alpha , \quad b_l^*\alpha=\nu_l \wedge \alpha, \quad b_l \alpha=\iota_{Z_l} \alpha.
\]
We perform the following computations at the center $x$ of this frame. It is easy to see that  the exterior derivative $d$ on one-forms can then be written:
\[
d=\sum_i a_i^* \nabla_{X_i} +\sum_l b_l^* \nabla_{Z_l}+\frac{1}{2} \sum_{i,j,l} \gamma_{ij}^l a_i^* a_j^* b_l
\]
and that
\[
\delta_{\mathcal{H}, \ve}=-\sum_i a_i \nabla_{X_i}+\sum_{i,j,l} \frac{1}{\ve} \gamma_{ji}^l a_i^* b_l^* a_j  -\gamma_{ji}^l a_i^* a_j b_l
\]
We can then proceed by direct computations, as in the proof of Proposition~\ref{prop:SLForms} and use Lemma 3.2 in \cite{BKW} to identify the term $\frac{1}{\varepsilon} \delta_\mathcal{H} T-\frac{1}{ \varepsilon}\mathbf{J}^2- \Ric_{\mathcal{H}}$.
\end{proof}

One recovers then immediately a result first proved in \cite{BKW}.

\begin{corollary}\label{Weitzenbock}
For every $f \in C^\infty(\M)$ and every $0<\varepsilon \le +\infty$, one has the following commutation:
\[
d\Dh f=\Dhe df.
\]
\end{corollary}

\begin{proof}
On functions, we  have for every $0<\varepsilon \le +\infty$,
\[
\Dh=-\delta_{\mathcal{H},\ve} d.
\]
So, the result follows from Proposition \ref{W3} and the fact that $d^2=0$.
\end{proof}

\subsection{Heat equation on one-forms}

In this section, we collect several analytic prerequisites. We will denote by $L^2(\wedge^1 \M, g_\ve)$ the $L^2$ space of one-forms with respect to the $(L^2,g_\varepsilon)$ product on sections.

\begin{proposition}
Let $0< \ve \le +\infty$. The operator $\Dhe$ is hypoelliptic.
\end{proposition}

\begin{proof}
We can locally write
\[
\Dhe=\sum_{i=1}^n (\nabla^\ve_{X_i} )^2 +V
\]
where $V$ is a first order real differential operator and $X_1,\cdots,X_n$ a local horizontal frame that satisfies the two-step H\"ormander's bracket generating condition. From H\"ormander's theorem (see \cite{Hor67} for the scalar case and for instance the arguments of Proposition 3.5.1 in \cite{Ponge} for the form case), one deduces the hypoellipticity of $\Dhe$.
\end{proof}

\begin{proposition}\label{Pol}
Let $0 < \ve <+\infty$. The operator $\Dhe$  is the generator of a strongly continuous and smoothing semigroup of bounded operators on $L^2(\wedge^1 T^*\M, g_\ve)$. Moreover, if $\{ e^{t\Dhe}, t \ge 0 \}$ denotes this semigroup, then for every smooth one-form $\alpha$, $\alpha_t =e^{t\Dhe} \alpha$ is the unique solution of the heat equation:
\begin{align*}
\begin{cases}
\frac{\partial \alpha_t}{\partial t}= \Dhe \alpha_t , \\
\alpha_0=\alpha .
\end{cases}
\end{align*}
\end{proposition}

\begin{proof}
We use arguments developed in \cite{BGS,MRT,Ponge}. First, since the manifold $\M$ is supposed to be compact, one deduces from \eqref{eq-L-form} that $\Dhe$ satisfies a G\aa rding type inequality: there exists a constant $K_\ve$ such that for every smooth one-form
\[
\langle \Dhe \alpha ,\alpha \rangle_{L^2,g_ve} \le K_\ve \| \alpha \|_{L^2,g_\ve}.
\]
As a consequence, $\Dhe$ is the generator of a strongly continuous  semigroup of bounded operators $e^{t\Dhe} \alpha$ on $L^2(\wedge^1 \M, g_\ve)$ (see \cite{P} pages 14 and 15 and Corollary page 81) that weakly solves the heat equation. Since $\Dhe$ is hypoelliptic, $e^{t\Dhe} $  actually admits a smooth heat kernel (as in chapter 5 of \cite{Ponge}) and thus $e^{t\Dhe} $ also strongly solves the heat equation. It remains to prove that solutions of the heat equation are unique. So, let $\alpha_t$ be a strong solution of
\begin{align*}
\begin{cases}
\frac{\partial \alpha_t}{\partial t}= \Dhe \alpha_t , \\
\alpha_0=0 .
\end{cases}
\end{align*}
Since $\M$ is compact the functional
\[
\phi(t)=\int_\M \| \alpha_t \|^2_{g_\ve} d\mu
\]
is well defined and one has
\[
\phi'(t)=2\int_\M \langle \Dhe \alpha_t ,\alpha_t \rangle_{g_\ve} d\mu \le 2 K_\ve \phi (t).
\]
So from Gronwall's lemma $\phi(t)=0$ and solutions of the heat equation are therefore unique.
\end{proof}

It should be noted that this semigroup $e^{t\Dhe} \alpha$ coincides with the semigroup constructed in \cite{GrTh16} (or \cite{Bau14,BKW} in the Yang-Mills case) and therefore admits a Feynman-Kac type representation. The next result is then an easy consequence of Corollary \ref{Weitzenbock} and Proposition \ref{Pol} (it is also pointed out in \cite{Bau14,BKW,GrTh16}).

\begin{lemma} 
Let $0 < \ve <+\infty$. For every $t \ge 0$, and $f \in C^\infty (\M)$,
\[
d e^{t\Dh} f =e^{t\Dhe} df
\]
\end{lemma}

\begin{proof}
Both sides are solutions of the heat equation
\begin{align*}
\begin{cases}
\frac{\partial \alpha_t}{\partial t}= \Dhe \alpha_t , \\
\alpha_0=df .
\end{cases}
\end{align*}
\end{proof}
\subsection{Horizontal semigroup  and $H_{dR}^1(\M)$}

In this section, in addition of the assumptions of the previous section, we now assume a positive curvature condition. if $\alpha$ is a one-form, we define a tensor $Q$ by
\begin{align}\label{def R}
\langle Q (\alpha),\alpha\rangle_g=\left\langle \Ric_{\mathcal{H}} (\alpha), \alpha \right\rangle_\mathcal{H} -\left \langle \delta_\mathcal{H} T (\alpha) , \alpha \right\rangle_\mathcal{V} -\frac{1}{4} \mathbf{Tr}_\mathcal{H} (J^2_{\alpha})
\end{align}
and assume in this section  that $Q $ is positive in the sense that there exists a positive constant $\lambda >0$ such that for every smooth and horizontal one-form $\alpha$, 
\begin{align}\label{cond 1}
\langle Q (\alpha) , \alpha \rangle_g \ge \lambda \| \alpha \|_g^2.
\end{align}
We note that $Q$ is only symmetric if the Yang-Mills condition $\delta_\mathcal{H} T=0$ is satisfied.
Our goal is to prove that, if $\ve$ is large enough, harmonic forms for $\Dhe$ are zero and deduce then that $H_{dR}^1(\M)=0$. The main tool is the heat semigroup $e^{t \Dhe}$, $t \ge 0$,  that was constructed in the previous subsection.

\begin{remark}\label{Riemann remark}
The tensor $Q$ was first introduced in \cite{BaudoinGarofalo2017} in the context of sub-Riemannian manifolds with transverse symmetries, and with the notations of \cite{BaudoinGarofalo2017}, one actually has $\mathcal{R}(f,f)=\langle Q (df) , df \rangle_g$. A direct computation (as Theorem 9.70, Chapter 9 in \textup{\cite{Besse}})  shows that the Ricci curvature (for the Levi-Civita connection) of $g_\ve$ is given by,
$$\Ric_{g_\ve}(v,w) = \left\{ \begin{array}{ll}
\Ric_\V(v,w) - \frac{1}{4\ve^2} \tr_\ch J_v J_w,  & \text{if $v, w \in \V$,} \\ \\
-\frac{1}{2\ve} \langle \delta_\mathcal{H} T(v), w \rangle_g & \text{if $v \in \ch$, $w \in \V$,} \\ \\
 \Ric_\ch(v,w)  +  \frac{1}{2\ve } \langle \mathbf{J}^2 w,  v \rangle_g & \text{if $v,w \in \ch$}
\end{array} \right.$$
with $\Ric_\V$ being the Ricci curvature of the leafs of the foliation. In particular, we see that for every $u \in \mathcal{H}, v \in \mathcal{V}$,
\begin{align}\label{adiabatic Q}
 \Ric_{g_\ve}(v+\ve w  ,v+\ve w) = \frac{1}{2\ve } \langle \mathbf{J}^2 v,  v \rangle_g + \langle Q (v+w) , v+w \rangle_g +\ve^2 \Ric_\V(w,w).
\end{align}
Observe that, due to the vertical curvature term $\Ric_\V$, our positivity assumption on $Q$ does not necessarily imply that there exists $\ve>0$ such that  $\Ric_{g_\ve }>0$. Therefore Theorem \ref{H1 1forms}, is not a consequence of the classical Hodge theorem, since we do not require any assumption on $\Ric_\V$.
\end{remark}

\begin{remark}
The condition \eqref{cond 1} can be simplified if the foliation is of Yang-Mills type, that is $\delta_\mathcal{H} T=0$.
Indeed, if $Z$ is a vertical vector field and $X_1,\cdots,X_n$ a local horizontal frame, it is easy to compute that
\[
-\mathbf{Tr}_\mathcal{H} (J^2_{Z})=\sum_{i,j=1}^n \langle T(X_i,X_j) ,Z \rangle^2=\sum_{i,j=1}^n \langle [X_i,X_j] ,Z \rangle^2.
\]
Therefore, the step-two generating condition on $\mathcal{H}$ and the compactness of $\M$ then implies that there exists a positive constant $a >0$ such that
\[
-\frac{1}{4} \mathbf{Tr}_\mathcal{H} (J^2_{\alpha}) \ge a \| \alpha_\mathcal{V} \|^2.
\]
Therefore, on Yang-Mills foliations  (like K-contact foliations), the condition \eqref{def}  is equivalent to the fact that  the  Ricci curvature of the Bott connection is positive on the horizontal bundle.
\end{remark}
\

We now turn to the proof of Theorem \ref{H1 1forms}. We have the following first lemma:
\begin{lemma}
There exist $\ve >0$ and a constant $c_\ve >0$ such that for every closed and smooth-one form    $\alpha$,
\begin{align}\label{estimate garding}
\frac{1}{2} \Dh \| \alpha \|_{\varepsilon}^2 -\langle \Dhe \alpha , \alpha \rangle_{\varepsilon}  \ge  c_\ve \| \alpha \|_\ve^2.
\end{align}
\end{lemma}
\begin{proof}
As it has been shown in \cite{BKW} (see also Theorem 4.7 in \cite{BaudoinEMS2014}), the formula \eqref{def} implies that for any smooth form $\alpha$,
\begin{align*}
\frac{1}{2} \Dh \| \alpha \|_{\varepsilon}^2 -\langle \Dhe \alpha , \alpha \rangle_{\varepsilon} & =  \| \nabla_{\mathcal{H}}^\ve \alpha  \|_{\varepsilon}^2 + \left\langle \Ric_{\mathcal{H}} (\alpha), \alpha \right\rangle_\mathcal{H} -\left \langle \delta_\mathcal{H} T (\alpha) , \alpha \right\rangle_\mathcal{V} +\frac{1}{\varepsilon} \langle \mathbf{J}^2 (\alpha) , \alpha \rangle_\mathcal{H} 
\end{align*}
A computation similar to the proof of Proposition 3.6 in \cite{BKW} shows then that if $\alpha$ is a closed one-form one has then
\[
 \| \nabla_{\mathcal{H}}^\ve \alpha \|_{\varepsilon}^2 \ge -\frac{1}{4} \mathbf{Tr}_\mathcal{H} (J^2_{\alpha}).
\]
Therefore,
\[
\frac{1}{2} \Dh \| \alpha \|_{\varepsilon}^2 -\langle \Dhe \alpha , \alpha \rangle_{\varepsilon} \ge -\frac{1}{4} \mathbf{Tr}_\mathcal{H} (J^2_{\alpha})  + \left\langle \Ric_{\mathcal{H}} (\alpha), \alpha \right\rangle_\ve - \left \langle \delta_\mathcal{H} T (\alpha) , \alpha \right\rangle_\V +\frac{1}{\varepsilon} \langle \mathbf{J}^2 (\alpha) , \alpha \rangle_\ve.
\]
When $\ve$ is large enough, the right hand side of the above inequality is bounded from below by $c_\ve \| \alpha \|_\ve^2$ which concludes the proof.
\end{proof}

In the remainder of the section, we fix $\ve >0$ so that \eqref{estimate garding} is true.
\begin{lemma} \label{lemma:Convergence}
Let $\alpha$ be a smooth and closed one-form, then in $L^2$, 
\[
\lim_{t \to +\infty} e^{t \Dhe} \alpha =0.
\]
\end{lemma}

\begin{proof}
Let $\alpha$ be a smooth and closed one-form and consider the functional
\[
\phi(t)= \int_\M \| e^{t \Dhe} \alpha \|^2_\ve d\mu.
\]
One has
\[
\phi'(t)=2 \int_\M \langle  \Dhe e^{t \Dhe} \alpha , e^{t \Dhe} \alpha \rangle_\ve d\mu.
\]
Since, for every $t \ge 0$, $e^{t \Dhe} \alpha$ is a closed one-form (it will be proved in Lemma \ref{dve inter} that $de^{t \Dhe} \alpha=Q_t d\alpha=0 $ where $Q_t$ is a semigroup on two-forms), one deduces from \eqref{estimate garding}  that
\[
\phi'(t) \le -2  c_\ve \phi (t).
\]
This implies
\[
\int_\M \| e^{t \Dhe} \alpha \|^2_\ve d\mu \le e^{- 2c_\ve t} \int_\M \|  \alpha \|^2_\ve d\mu \rightarrow_{t \to +\infty} 0.
\]
\end{proof}

We now have:

\begin{lemma} \label{lemma:ClosedToClosed}
Let $\alpha$ be a smooth closed one-form, then for every $t \ge 0$, $e^{t \Dhe} \alpha-\alpha$ is an exact smooth form.
\end{lemma}
\begin{proof}
Let $\alpha$ be a smooth closed form. We have then:
\begin{align*}
e^{t\Dhe} \alpha -\alpha & = \int_0^t \frac{d}{ds} e^{s\Dhe} \alpha ds = \int_0^t  \Dhe e^{s\Dhe} \alpha ds \\
&=\int_0^t   e^{s\Dhe} \Dhe \alpha ds = \int_0^t   e^{s\Dhe} d \delta_{\mathcal H,\ve} \alpha ds =d \left( \int_0^t   e^{s\Dh}  \delta_{\mathcal H,\ve}  \alpha ds  \right).
\end{align*}
Therefore $e^{t\Dhe} \alpha -\alpha$ is an exact smooth form. 
\end{proof}

We can now finally conclude the proof of Theorem \ref{H1 1forms}.

\begin{proof}
Let $\alpha$ be a smooth closed one-form.  We denote by $[\alpha] \in H_{dR}^1(\M)$ its de Rham cohomology class. From the previous lemma, for any $t \ge 0$, one has $[e^{t\Dhe} \alpha]=[\alpha]$. Since, in $L^2$, $e^{t\Dhe} \alpha \to 0$, and from Hodge theory $\alpha \to [\alpha]$ is continuous in $L^2$, one concludes $[\alpha]=0$.
\end{proof}

\begin{remark}\label{optimality}
One can prove as in Proposition 3.2 in \cite{B},  that for any fiberwise linear map $\Lambda$ from the space of two-forms into the space of one-forms, and for any $x \in \M$, we have
\begin{align*}
 & \inf_{\eta, \| \eta (x) \|_{\varepsilon}=1} \left(  \frac{1}{2} (\Dh \| \eta \|_\varepsilon^2)(x) -\langle (\Delta_{\mathcal{H},\infty} +\Lambda \circ d )\eta (x) , \eta (x)\rangle_\varepsilon \right)  \\
 \le & \inf_{\eta, \| \eta (x) \|_{\varepsilon}=1} \left(  \frac{1}{2} (\Dh  \| \eta \|_\varepsilon^2)(x) -\left\langle \Delta_{\mathcal{H},2\ve} \eta (x) , \eta (x) \right\rangle_\varepsilon \right)\\
 =&\inf_{\eta, \| \eta (x) \|_{\varepsilon}=1} \left(-\frac{1}{4} \mathbf{Tr}_\mathcal{H} (J^2_{\eta})  + \left\langle \Ric_{\mathcal{H}} (\eta), \eta \right\rangle_\ve - \left \langle \delta_\mathcal{H} T (\eta) , \eta \right\rangle_\V +\frac{1}{\varepsilon} \langle \mathbf{J}^2 (\eta) , \eta \rangle_\ve\right)
\end{align*}
As a consequence, assume that $L$ is a Laplacian on forms which has the same symbol as $-\nabla_\mathcal{H}^* \nabla_\mathcal{H}$ and that satisfies $d \Dh=Ld$. Assume moreover that  there exist $\ve >0$ and a  constant $c_\ve>0$ such that for every  smooth-one form    $\alpha$,
\begin{align}\label{estimate garding}
\frac{1}{2} \Dh \| \alpha \|_{\varepsilon}^2 -\langle L \alpha , \alpha \rangle_{\varepsilon}  \ge  c_\ve \| \alpha \|_\ve^2.
\end{align}
Then, one must have $\langle Q (\alpha) , \alpha \rangle_g \ge c_\ve \| \alpha \|_\ve^2$.  This means that $Q$ is canonical and the optimal one to consider when applying the Bochner's method. In particular, Corollary 3.3.14 in \cite{EJL10} is a corollary (in our framework) of Theorem \ref{H1 1forms}.
\end{remark}

\section{Sub-Laplacians on forms and transverse Weitzenb\"ock formulas}

In the next Section 4, we will generalize the horizontal Bochner's method developed in the previous section  to all $k$-forms. The first step in doing so is to construct the relevant horizontal Laplacians. 
In this section we show how to construct such horizontal Laplacians. We propose here a very general construction in the setting of sub-Riemannian manifolds, generalizing \cite{GrTh16}, which shall later be applied to the special case of totally geodesic foliations.

\subsection{Sub-Laplacians on sub-Riemannian manifolds}
We will first consider operators on forms on a sub-Riemannian manifold $(\M, \ch, g_\ch)$. We define the corresponding cometric $g_\ch^*$ such that if $v_1, \dots, v_n$ is an orthonormal basis of $\ch_x$ and $\alpha_1, \alpha_2 \in T^*_x \M$, then
$$\langle \alpha_1, \alpha_2 \rangle_{g_\ch^*} = \sum_{i=1}^n \alpha_1(v_i) \alpha_2(v_i).$$
Notice that $g_\ch^*$ is degenerate along the subbundle of forms vanishing on $\ch$. For any two-tensor $\xi \in \Gamma^\infty(T\M^{\otimes 2})$, we define $\tr_{\ch} \xi(\times, \times) = \xi(g_\ch^*)$.

Let $\nabla$ be an arbitrary connection on $\M$ with Hessian $\nabla^2_{X,Y} = \nabla_X \nabla_Y - \nabla_{\nabla_X Y}$ for every $X, Y \in \Gamma^\infty(TM)$.
Consider the corresponding horizontal Laplacian
\begin{align}\label{covariant laplacian}
L_\ch : = \tr_{\ch} \nabla_{\times, \times}^2,
\end{align}
therefore defined in a local horizontal frame $X_1,\cdots,X_n$ by $L_\ch=\sum_{i=1}^n \nabla_{X_i} \nabla_{X_i} - \nabla_{\nabla_{X_i} X_i}$.

Let $\Omega = \Omega(\M)  = \oplus_{k=0}^{n+m} \Omega^k$ be the graded exterior algebra of smooth forms on $\M$. We want to define an operator $\Delta_\ch$ on forms satisfying the following properties.
\begin{enumerate}[\rm (I)]
\item \label{item:i} For any function $f \in \Omega^0$, we have $\Delta_\ch f = L_\ch f$.
\item \label{item:ii} For any form $\alpha \in M$, we have $\Delta_\ch d \alpha = d \Delta_\ch \alpha$.
\item \label{item:iii} The operator is of ``Weizenb\"ock-type'', meaning that,
$$\Delta_\ch = L_\ch - C,$$
where the operator $C$ has order zero as a differential operator.
\end{enumerate}
The case of $\Omega^0$ and $\Omega^1$, this was studied in \cite{Bau14,BKW,GrTh16}. In particular, from \cite[Prop 2.4]{GrTh16}, we know the following result.
\begin{proposition} \label{prop:Adjoint}
Let $T$ be the torsion of $\nabla$. Define a connection $\hat{\nabla}$ by
$$\hat{\nabla}_X Y  = \nabla_X Y - T(X,Y).$$
Then there exists an operator $\Delta_\ch$ on $\Omega^0 \oplus \Omega^1$ satisfying \eqref{item:i}, \eqref{item:ii} to \eqref{item:iii} if and only if $\hat{\nabla} g^*_\ch = 0$.
\end{proposition}
Using the terminology of \cite{EJL99}, the connection $\hat{\nabla}$ is called \emph{the adjoint connection} of~$\nabla$ (see the Appendix). It has torsion $-T$ and clearly $\nabla$ is the adjoint of~$\hat{\nabla}$. The condition $\hat{\nabla} g^*_\ch = 0$ is equivalent to $\ch$ being preserved under parallel transport and that
$$Z \langle X, Y \rangle_{g_\ch} = \langle \hat{\nabla}_Z X, Y \rangle_{g_\ch} + \langle X, \hat{\nabla}_Z Y \rangle_{g_\ch},$$
for any $Z \in \Gamma^\infty(TM)$ and $X, Y \in \Gamma^\infty(\ch)$. We will show that the result of Proposition~\ref{prop:Adjoint} generalizes to $k$-forms.

A Weitzenb\"ock formula for the sub-Laplacian is essentially found in \cite[Theorem 2.4.2]{EJL99}. We want to make a different presentation of this result, using a practical unified notation related to the Fermion calculus in the proof of Proposition~\ref{W3}. We will also include a proof, showing that one does not need to assume that the connection we are working with is of Le Jan-Watanabe type as in \cite{EJL99}.

Let $\Psi = \oplus_{i,j=1}^{n+m} \Psi_{(i,j)}$ denote the space of all smooth sections of
$$\wedge^\bullet T^*\M \otimes \wedge^\bullet T\M = \bigoplus_{i,j=0}^{n+m} \wedge^i T^*\M \oplus \wedge^j T\M. $$
For any element $\nu \in \Psi $ introduce the corresponding linear operator $C_\nu: \wedge^\bullet T^* \M \to \wedge^\bullet T^* \M $ by the following rules:
\begin{enumerate}[(i)]
\item For any $ \nu_1, \nu_2 \in \Psi$ and $f \in \Psi_{(0,0)}$,
$$C_{\nu_1 + \nu_2} = C_{\nu_1} + C_{\nu_2}, \quad C_{f \nu_1} = f C_{\nu_1} \quad \text{ and } \quad  C_1 = \id;$$ 
\item For a form $\alpha \in \Psi_{(i,0)}$, $i \geq 0$, we have $C_\alpha = \alpha \wedge$, the exterior product with $\alpha$;
\item For a vector field $X \in \Psi_{(0,1)}$, we have $C_X =\iota_X$, the contraction by $X$;
\item For any vector field $X \in \Psi_{(0,1)}$ and multi-vector field $\chi \in \Psi_{(0,j)}$, $j \geq 0$, we have $C_{X \wedge \chi} = C_{\chi} C_X$;
\item If $\alpha \in \Psi_{(i,0)}$ and $\chi \in \Psi_{(0,j)}$, $i \geq 0$, $j \geq 0$, then $C_{\alpha \otimes \chi} = C_\alpha C_\chi$.
\end{enumerate}
In other words, if $\nu = f \alpha_1 \wedge \dots \wedge \alpha_i \otimes X_1 \wedge \cdots \wedge Y_j$, then
$$C_\nu(\eta) = f \alpha_1 \wedge \cdots \wedge \alpha_i \wedge ( \iota_{Y_j} \cdots \iota_{X_1}\eta).$$
\begin{example}
Let $S\in \Gamma^\infty( \End(T^*\M ))$ be a vector bundle endomorphism, which we can consider as an element in $\Psi_{(1,1)}$. Then for one-forms $\alpha_1, \alpha_2 \dots, \alpha_i$
\begin{align*}C_S( \alpha_1 \wedge \alpha_2 \cdots \wedge \alpha_i )& = S(\alpha_1) \wedge \alpha_2 \wedge \cdots \wedge \alpha_i + \alpha_1 \wedge S(\alpha_2) \wedge \cdots \wedge \alpha_i \\
& \qquad + \cdots + \alpha_1 \wedge \alpha_2 \wedge \cdots \wedge S(\alpha_i),\end{align*}
\end{example}

Using the above notation notation, and denoting by $\hat{R}$ the curvature tensor of $\hat{\nabla}$ we have the following result. 

\begin{proposition} \label{prop:SLForms}
Let $L_\ch$ be the horizontal Laplacian defined relative to a connection $\nabla$ as in \eqref{covariant laplacian}. Assume that its adjoint connection  $\hat{\nabla}$ satisfies $\hat{\nabla} g_\ch^*=0$. Introduce the operator $\delta_\ch$ on forms by
$$\delta_\ch\eta = - \tr_\ch (\nabla_\times \eta)(\times, \, \cdot \, ).$$
Then the operator
$$\Delta_\ch = - d \delta_\ch - \delta_\ch d,$$
satisfies \eqref{item:i}-\eqref{item:iii}. In fact, if we define $\Ric = \Ric_{1,1} + \Ric_{2,2}$, $\Ric_{i,j} \in \Psi_{(i,j)}$, by 
\begin{equation} \label{RicTerms} \Ric_{1,1}(v) = - \tr_\ch \hat{R}(\times, v) \times, \qquad (\alpha \wedge \beta) \Ric_{2,2}(v,w) = 2 \langle \hat{R}(v,w) \alpha, \beta \rangle_{g_\ch^*},\end{equation}
the operator $\Delta_\ch$ satisfies the Weitzenb\"ock identity
$$\Delta_\ch = L_\ch - C_{\Ric}.$$
\end{proposition}
Before completing the proof, we emphasize the following properties of our creation/annihilation operators. Let $\nabla$ be an arbitrary connection with torsion $T$.
\begin{enumerate}[\rm (a)]
\item For any $X \in \Psi_{(0,1)}$ and $\nu \in \Psi_{(i,j)}$, observe that
$$\nabla_X C_\nu = C_{\nabla_X \nu} + C_\nu \nabla_X, \qquad C_X C_\nu = C_{\iota_X \nu} + (-1)^{i-j} C_\nu C_X.$$
\item If $\nu \in \Psi_{(i,j)}$, then $C_\nu : \Omega^k \to \Omega^{k+i-j}$ and $C_\nu \Omega^k = 0$ for any $k < i$.
\item Let $Z_1$, $\dots$, $Z_{n+m}$ be an arbitrary local frame of tangent bundle and let $\beta_1$, $\dots$, $\beta_{n+m}$ be the corresponding coframe. Observe that we can then write the exterior differential as
\begin{equation} \label{ExteriorD} d = C_T + \sum_{i=1}^{n+m} C_{\beta_i} \nabla_{Z_i} = \frac{1}{2} \sum_{i,j=1}^{n+m} \beta_i \wedge \beta_j \wedge \iota_{T(Z_i, Z_j)} + \sum_{i=1}^{n+m} \beta_i \wedge \nabla_{Z_i} 
\end{equation}
This equality follows by observing that if we write $d'$ for the expression on the right hand side of \eqref{ExteriorD}, then $d'(\alpha \wedge \beta) = (d' \alpha ) \wedge \beta + (-1)^k \alpha \wedge d' \beta$ for any form $\beta$ and $k$-form $\alpha$. We also have $d'(f \alpha) = df \wedge \alpha + f d' \alpha$. Hence, $d = d'$ if the operators agree on one-forms, and it is simple to verify that indeed
$$d\alpha(X,Y) = \nabla_X \alpha(Y) - \nabla_Y \alpha(X) + \alpha(T(X,Y)).$$
\item Let $g$ be a Riemannian metric with corresponding identification $\sharp : \bigwedge^{\bullet} T^*\M \to \bigwedge^{\bullet} T\M$ and $\flat :\bigwedge^{\bullet} T\M \to \bigwedge^{\bullet} T^*\M$ and define $ \Psi \to \Psi$, $\nu \mapsto \nu^*$ as the vector bundle map determined by $(\alpha \otimes \chi)^* = \flat \chi \otimes \sharp \alpha$. Then $C^*_\nu = C_{ \nu^*}$.
\end{enumerate}
\begin{proof}[Proof of Proposition~\ref{prop:SLForms}] Let $T$ denote the torsion of $\nabla$. Relative to $\nabla$ and for any $X, Y \in \Gamma^\infty(T\M)$ define operators
$$\delta[X \otimes Y] = - C_X \nabla_Y, \qquad L[X \otimes Y] = \nabla^2_{Y,X}.$$
We extend this to an arbitrary section of $\chi \in \Gamma^\infty(TM^{\otimes 2})$ by linearity and define $\Delta[\chi] = -d \delta[\chi] - \delta[\chi]d $. Clearly, $\Delta[\chi]$ commutes with $d$.

Let $Z_1, \dots, Z_{n+m}$ be a local basis of $TM$ with corresponding coframe $\beta_1, \dots, \beta_{n+m}$. If $x \in \M$ is an arbitrary point, we may choose this basis such that such that $\nabla Z_i(x) = 0$. Evaluating at the point~$x$, we obtain
\begin{align*}
\Delta[X \otimes Y] &= L[X \otimes Y] - \sum_{i=1}^m C_{\beta_i} C_{X} \nabla_{Y, Z_i}^2 + \sum_{i=1}^m C_{\beta_i} C_{\nabla_{Z_i} X} \nabla_Y \\
& \quad + \sum_{i=1}^m C_{\beta_i} C_{X} \nabla_{Z_i, Y}^2  + \sum_{i=1}^m C_{\beta_i} C_{X} \nabla_{\nabla_{Z_i} Y}  + C_X C_{\nabla_Y T} + (C_X C_T + C_T C_X) \nabla_Y
\end{align*}
We will use the identities $C_X C_T = - C_T C_X +C_{T(X, \, \cdot \,)}$ and $\nabla_{X,Y}^2 - \nabla_{Y,X}^2 = C_{R(X,Y)} - \nabla_{T(X,Y)} $ for the result. Note that in the last formula, $R(X,Y)$ denotes the curvature endomorphism acting on the cotangent bundle, i.e. the element in $\Psi_{1,1}$ determined by
$$R(X,Y): (\alpha , v) \mapsto (R(X,Y) \alpha)(v) = - \alpha(R(X,Y) v), \qquad \alpha \otimes v \in T^*M \otimes TM.$$
We compute
\begin{align} \nonumber
\Delta[X \otimes Y] &= L[X \otimes Y] - \sum_{i=1}^m C_{\beta_i} C_{X} C_{R(Y, Z_i)} + \sum_{i=1}^m C_{\beta_i} C_{X} \nabla_{T(Y,Z_i)}  \\ \nonumber
& \quad + \sum_{i=1}^m C_{\beta_i} C_{\nabla_{Z_i} X} \nabla_Y   + \sum_{i=1}^n C_{\beta_i} C_{X} \nabla_{\nabla_{Z_i} Y}  - C_{\nabla_Y T} C_X - C_{\delta[X \otimes Y] T} + C_{T(X, \cdot)} \nabla_Y \\ \nonumber
 &= L[X \otimes Y] + C_{R(Y, \cdot)X}  - \sum_{i=1}^m C_{\beta_i} C_{R(Y, Z_i)} C_X +  C_{(T(X, \cdot) +\nabla X)} \nabla_Y   \\ \nonumber
 &\quad + \sum_{i=1}^m C_{\beta_i \otimes X} \nabla_{\nabla_{Z_i} Y+ T(Y, Z_i)}  - C_{X \wedge \nabla_Y T}  - C_{\delta[X \otimes Y] T} . \nonumber
\end{align}
Introduce $\Ric[X \otimes Y] = \Ric_{1,1}[X \otimes Y] + \Ric_{2,2}[X\otimes Y]$ defined by
\begin{equation} \label{Ric} \left\{ \begin{array}{lcl}
\Ric_{1,1}[X \otimes Y](v) & = & - \hat{R}(Y, v) X,  \\
\Ric_{2,2}[X \otimes Y](v,w) & = & X \wedge \hat{R}(v, w) Y. \end{array}
\right. \end{equation}
It then follows from \eqref{Curv1} and \eqref{Curv2} that
\begin{align*} \Delta[X \otimes X] &=  L[X \otimes X] - C_{\Ric[X \otimes X]}+  C_{\hat{\nabla} X} \nabla_X+ \sum_{i=1}^m C_{\beta_i \otimes X} \nabla_{\hat{\nabla}_{Z_i} X}. \end{align*}

Define $\Delta[g_\ch^*] = \Delta_\ch$, $L[g^*_\ch] = L_\ch$ and $\Ric[g^*_\ch] = \Ric = \Ric_{1,1} +\Ric_{2,2}$. Assume that $\hat{\nabla} g^*_\ch = 0$. Then for every $x$, there is an orthonormal basis $X_1, \dots, X_{n}$ of the horizontal bundle $\ch$ so that $\hat{\nabla} X_i(x) = 0$. The result follows.
\end{proof}

\subsection{Taming metrics}

Let $\nabla$ be a connection on $(\M, \ch, g_\ch)$ such that its adjoint connection $\hat{\nabla}$ satisfies $\hat{\nabla} g_\ch^* = 0$. In order to have an $L^2$-inner product of forms and perform the Bochner's method, we introduce a Riemannian metric $g$ such that $g|\ch = g_\ch$. Such a Riemannian metric is said to \emph{tame} $g_\ch$. If we assume that this metric is compatible with our original connection $\nabla$, this turns out to be a surprisingly restrictive requirement.
\begin{proposition}[\cite{GrTh16}]
Let $\nabla$ be any connection with adjoint connection $\hat{\nabla}$. Assume that $\hat{\nabla} g_\ch^* =0$. Then there exists a Riemannian metric $g$ such that
$$g|\ch = g_\ch, \qquad \text{ and } \qquad \nabla g =0,$$
if and only if
\begin{equation} \label{TG} (\calL_X g)(Z,Z) = 0, \qquad \text{ and } \qquad (\calL_Z g)(X,X) =0.\end{equation}
for any $X \in \Gamma^\infty(\ch)$ and $Z \in \Gamma^\infty(\ch^\perp)$.
\end{proposition}

For the special case when $\V := \ch^\perp$ is integrable, the condition in \eqref{TG} is equivalent to the foliation $\mathcal{F}$ being Riemannian, bundle-like, with totally geodesic leaves. For more details, see \cite{GrTh16}.

\section{Horizontal Bochner's method on $k$-forms}

In this Section, we now perform the horizontal Bochner's method on  $k$-forms, $k \ge 2$, in the setting of Section 2, using the horizontal Laplacians constructed in the previous Section.  So, the framework is the following. Let $\M$ be a smooth, oriented, connected, compact manifold with dimension $n+m$. We assume that $\bM$ is equipped with a Riemannian foliation $\mathcal{F}$ with bundle-like complete metric $g$ and totally geodesic  $m$-dimensional leaves such that $m<n$.  We also assume that the horizontal distribution $\mathcal{H}$ is two-step bracket generating. As before in Section 2, we consider  the one-parameter family of Riemannian metrics defined by:
\[
g_{\varepsilon}=g_\mathcal{H} \oplus  \frac{1}{\varepsilon }g_{\mathcal{V}}, \quad \varepsilon >0,
\]
and introduce the family of connections 
\[
\nabla^\varepsilon_X Y= \nabla_X Y -T(X,Y) +\frac{1}{\varepsilon} J_Y X, \quad 0 < \ve \le +\infty.
\] 
The corresponding horizontal Laplacian is then
\[
\Dhe=-d \delta_{\ch,\ve}-\delta_{\ch,\ve} d,
\]
and it satisfies the Weitzenb\"ock identity given in Proposition \ref{prop:SLForms}, that is we have
\[
\Dhe=-(\nabla_\mathcal{H}^\ve)^* \nabla_\mathcal{H}^\ve -C_{\Ric^\ve}.
\]

\begin{remark}
If $\alpha$ is a basic form, then  $\Delta_{\mathcal{H},\infty} \alpha$ is also basic and actually coincides with the so-called basic Laplacian of the foliation. We refer to  \cite{MRT,Park} for a study of the basic Laplacian and of its connection to basic cohomology.
\end{remark}

\begin{remark}
As in the one-form case, the operator $C_{\Ric^\ve}$ is not symmetric in general.
\end{remark}

\subsection{Main result}

Our goal in this section is to prove the following result:

 \begin{theorem} \label{th:Hk0}
Let $R$ be the curvature of the Bott connection. Define the horizontal curvature operator $R_\ch: \wedge^2 \mathcal{H}^* \to \wedge^2 \mathcal{H}^*$ as the linear operator determined by equation
$$R_\ch(\beta_1 \wedge \beta_2)(v,w) = \langle R(\sharp^\ch \beta_1, \sharp^\ch \beta_2)w, v\rangle_\ch, \qquad \beta_1, \beta_2 \in \ch^*, v,w \in \ch.$$
Assume that for some constant~$c >0$,
\begin{equation} \label{PositiveRH}
\langle R_\ch(\alpha) ,\alpha \rangle_{\ch} \geq c \| \alpha \|^2_\ch, \qquad \text{for any $\alpha \in \wedge^2 \ch^*$}.
\end{equation}
Assume furthermore, that for any $Z \in \Gamma^\infty(\V)$, we have $\nabla_Z T = 0$. Then $H_{dR}^k(\M) = 0$ for every $m < k < n $.
\end{theorem}

To understand the restriction $m < k <n$, recall that in Riemannian geometry, a positive curvature operator implies that $H^k_{dR}(\M) = 0$ whenever $0 < k < n = \dim \M$, see \cite{Mey71,GaMe73}, as the operator $C_{\Ric}$ will be strictly positive for $k$-forms with $k \neq 0, n$. We will similarly show that, for sufficiently large value of $\ve >0$, a positive horizontal curvature operator implies that $C_{\Ric^\ve}$ is strictly positive on forms $\alpha \wedge \beta$ where $\alpha$ is an $i$-form with $i \neq 0,n$ vanishing on $\V$ and an arbitrary form $\beta$ vanishing on $\ch$. In particular, it will be positive for $k$-forms with $m < k <n$. %Our proof of the theorem will be divided into several lemmas. We first discuss  a straightforward corollary.

%\begin{example}
%Let $(\mathbb{B}, g^{\mathbb{B}})$ be a compact $n$-Riemannian manifold with positive curvature operator $\wedge^2 T^*\mathbb{B} \to 
%\wedge^2 T^* \mathbb{B}$,
%$$\alpha \wedge \beta \mapsto \langle R^{\mathbb{B}}(\sharp \beta, \sharp \alpha) \, \cdot \, , \, \cdot \, \rangle_{\mathbb{B}}.$$
%Let $A \to \M \stackrel{\pi}{\to} \mathbb{B}$ be a principal bundle over $\mathbb{B}$ with a compact abelian structure group $A$ of dimension $m$. We assume that there is a connection form $\omega$ with a curvature form $\Omega$ that is surjective. In other words, $\mathcal{H} = \ker \omega$ is bracket-generating of step~$2$. Then $H^k_{dR}(\M) =0$ for $m < k < n$. To see this, choose an inner product $\langle \, \cdot \, , \, \cdot \, \rangle_{\mathfrak{a}}$ on the Lie algebra $\mathfrak{a}$ of~$A$. If we define
%$$\langle v, w\rangle_g = \langle \pi_* v, \pi_* w \rangle_{\mathbb{B}} + \langle \omega(v), \omega(w) \rangle_{\mathfrak{a}},$$
%then $\{ \pi^{-1}(y) \, : \, y \in \mathbb{B} \}$ is a totally geodesic foliation, and if $\nabla$ is the Bott connection with torsion~$T$, we have $\nabla_Z T = 0$ for any $Z \in \Gamma^\infty(\V)$. Furthermore, the horizontal curvature operator $R_\ch$ is positive by the positivity of the curvature operator on $\mathbb{B}$, see \cite[Proposition~3.4]{GrTh16a}.
%\end{example}

By the Bianchi identity \eqref{Bianchi}, note that the condition $\nabla_Z T=0$ for any $Z \in \Gamma^\infty(\V)$ is equivalent to the condition $R(\pi_\ch \, \cdot \, ,\pi_\ch \, \cdot \, ) Z = 0$ since
\begin{equation} \langle (\nabla_Z T) (X,Y) , W \rangle_g = \langle R(X, Y) Z, W \rangle_g\end{equation}
for any $X, Y \in \Gamma^\infty(\ch)$ and $Z,W \in \Gamma^\infty(\V)$. In particular, we have have $\langle (\nabla_Z T) (X,Y) , Z \rangle_g = 0$ since $\nabla$ preserves the metric $g$. As a consequence, the condition $\nabla_Z T = 0$ for any $Z \in \Gamma(\V)$ is always fulfilled whenever $\calV$ is one-dimensional. From this we obtain the following corollary.

\begin{corollary}
 Let $\calF$ have leaves of dimension $1$. Assume furthermore that both the curvature operator $R_\ch$ and the operator $Q$ on one-forms has positive lower bounds. Then $H_{dR}^k(\M) = 0$ for every $k \neq 0,n+1 = \dim \M$. In particular, for oriented K-contact foliations, \eqref{PositiveRH} alone implies that $H_{dR}^k(\M) = 0$ for all $0 < k < n +1 = \dim \M$.
\end{corollary}

\begin{proof}
Since $R_\ch > 0$ and $\V$ has rank~1, we obtain that $H^k_{dR}(\M) = 0$ for $1 < k < n$ from Theorem~\ref{th:Hk0}. Furthermore, since $Q > 0$, we obtain $H^1_{dR}(\M) = 0$ by Theorem~\ref{H1 1forms}, which also gives us $H^n_{dR}(\M) = 0$ by Poincar\'e duality.
As for the special case of oriented $K$-contact manifolds, $R_\ch > 0$ implies $Q > 0$ since  $\delta_\ch T =0$.
\end{proof}

\begin{remark}
Similarly to Remark \ref{Riemann remark}, we can again compare this result with the result that can be obtained from Riemannian geometry of $(\M, \calF, g)$. For simplicity, we discuss  the case when the leaves of $\calF$ have dimension $m =1$. Let $Z$ be a local unit length basis vector relative to $g$. Let $\alpha$ be a two-form, define $\beta = \iota_Z \alpha$ and
$$\alpha^0 = \alpha - \flat Z \wedge \beta = \frac{1}{2} \sum_{i,j=1}^n \alpha_{ij}^0 \flat X_i \wedge \flat X_j.$$
If $R_\ve$ is the curvature operator of $R^{g_\ve}$, then, after a long computation,  we have that
\begin{align*}
& \langle R_\ve \alpha, \alpha \rangle_\ve = \langle R_{\ve}(\alpha^0+ \flat Z \wedge \beta), \alpha^0 +  \flat Z \wedge \beta \rangle_\ve \\
& = \langle R_\ch (\alpha^0) , \alpha^0 \rangle_\ve -\frac{1}{\ve} q(\alpha^0, \alpha^0) - \langle \flat (\nabla J)_Z \sharp \beta, \alpha^0  \rangle_g  - \frac{\ve}{4} \langle \mathbf{J}^2 \sharp \beta, \sharp \beta \rangle_\ve 
\end{align*}
where
\begin{align*}
q(\alpha^0, \alpha^0) & := \frac{1}{8} \sum_{i,j,r,s=1}^n \alpha_{ij}^0 \alpha_{rs}^0  \langle T(X_i,X_j), T(X_r, X_s) \rangle_g \\
& \quad + \frac{1}{16} \sum_{i,j,r,s=1}^n \alpha_{ij}^0 \alpha_{rs}^0 \langle T(X_i, X_r), T(X_j, X_s) \rangle_g + \frac{1}{16} \sum_{i,j,r,s=1}^n \alpha_{ij}^0 \alpha_{rs}^0 \langle T(X_i, X_s), T(X_r, X_s) \rangle_g .
 \end{align*}
Define
$$\begin{array}{ll}
c = \min \{ \langle R_\ch \alpha^0  , \alpha^0 \rangle_g \, : \, \| \alpha^0 \|_g =1 \}  > 0, & k = \min \{ - \langle \mathbf{J}^2 X, X \rangle_g : \| X\|_g = 1 \} \\
M_q = \max  \{ q(\alpha^0  , \alpha^0) \, : \, \| \alpha^0 \|_g =1 \} < \infty,  \quad & M_{\nabla J} = \max  \{ \langle \flat (\nabla J)_Z X, \alpha \rangle_g \, : \, \|X|| = 1,\| \alpha^0 \|_g =1 \} < \infty 
\end{array}
$$
Then
\begin{align*}
& \langle R_{\ve}(\alpha^0+ \flat Z \wedge \beta), \alpha^0 +  \flat Z \wedge \beta \rangle_\ve \\
& \geq c \| \alpha^0 \|^2_\ve - \frac{1}{\ve} M_q \| \alpha^0 \|_\ve^2 - \frac{M_{\nabla J} }{\sqrt{\ve}} \| \flat Z \wedge \beta \|_\ve  \|\alpha^0 \|_\ve + \frac{k}{4} \| \flat Z \wedge \beta\|^2_\ve.
\end{align*}
Hence, as long as $k > 0$, we will have a positive lower bound of $R_\ve$ for sufficiently large values of $\ve$.  However, at our level of generality, we may actually have $k=0$, unless $\ch $ satisfies the property that for every non-zero horizontal vector field  $X$, $T\M$ is Lie generated by $\ch $ and $[\ch , X]$ .
\end{remark}

\subsection{The horizontal Ricci curvature operator on forms}

For any $k \geq 0$, we define an orthogonal decomposition $\wedge^k T^* \M = \oplus_{i+j=k} \wedge^{(i,j)} T^*\M$, where $\wedge^{(i,j)} T^* \M$ is spanned by all elements that can be written as $\alpha \wedge \beta$ where $\alpha$ and $\beta$ are respectively an $i$-form vanishing on $\V$ and a $j$-form vanishing on $\ch$. Notice that if we define an inner product $\langle \, \cdot \, , \, \cdot \, \rangle_\ve$ with respect to $g_\ve$ on covectors, and if $\alpha$ and $\beta$ are sections of $\wedge^{(i,j)} T^* \M$, then
$$\langle \alpha, \beta \rangle_\ve = \ve^j \langle \alpha, \beta \rangle_g.$$

 Define $\Ric_\ch = \Ric^\infty$. In other words, $\Ric_{\ch} = (\Ric_\ch)_{1,1} + (\Ric_\ch)_{2,2}$, where the terms are defined as in \eqref{RicTerms} with respect to the curvature of the Bott connection. We have then the following result:

\begin{proposition}\label{Cric}
 Assume that the horizontal curvature operator $R_\ch$ has a positive lower bound $c$ as in \eqref{PositiveRH}. Then, there is a constant $c_1>0$ such that for any $\alpha \in \wedge^{(i,j)} T^* \M,$ $i \neq 0, n$ and any $\ve > 0$,
$$\langle C_{\Ric_\ch} \alpha, \alpha \rangle_\ve \geq c_1 \| \alpha\|^2_\ve.$$ 
\end{proposition}

The proof of this proposition relies on the following lemma:

\begin{lemma} \label{lemma:HorCurvOp}
 For $k =1,2$, let $\alpha_k$ be an $i$-form vanishing on $\V$ and let $\beta_k$ be a $j$-form vanishing on $\ch$. Then
$$\langle C_{\Ric_\ch} (\alpha_1 \wedge \beta_1), \alpha_2 \wedge \beta_2 \rangle_\ve = \langle C_{\Ric_\ch} \alpha_1, \alpha_2 \rangle_\ve \langle \beta_1 , \beta_2 \rangle_\ve = \langle  \alpha_1, C_{\Ric_\ch} \alpha_2 \rangle_\ve \langle \beta_1 , \beta_2 \rangle_\ve  .$$
\end{lemma}

\begin{proof}
We first observe that $(\Ric_\mathcal{H})_{1,1}$ vanishes on $\V$ and has its image in $\ch$ by \cite[Lemma~3.3 (b)]{GrTh16a}. For $(\Ric_\mathcal{H})_{2,2}$, observe first that if $X_1, X_2 \in \Gamma^\infty(\ch)$ and $Z,W \in \Gamma^\infty(\V)$, then
$$\langle R(Z,W) X_1, X_2 \rangle_g = \langle \circlearrowright R(Z,W) X_1 , X_2 \rangle_g = \langle \circlearrowright (\nabla_Z T)(W,X_1), X_2 \rangle_g = 0.$$
If $X_1, X_2, W \in \Gamma(\ch)$ and $Z \in \Gamma(\V)$, we use the first Bianchi identity \eqref{Bianchi} and the relation \eqref{CommuteBott}, Appendix~A. Define
$$A(X,Y) : = T(X,Y) - J_X Y - J_Y X.$$
We then have relation
\begin{align*} 
& \langle R(Z,W)X_1, X_2 \rangle_g  \\ 
& = \frac{1}{2} \langle (\nabla_Z A)(W,X_1), X_2 \rangle_g - \frac{1}{2} \langle (\nabla_{W} A)(Z,X_1), X_2 \rangle_g \\ 
& \qquad - \frac{1}{2} \langle (\nabla_{X_1} A)(X_2,Z) , W \rangle_g  +  \frac{1}{2} \langle (\nabla_{X_2} A)(X_1,Z),W \rangle_g \\
& = \frac{1}{2} \langle Z, \circlearrowright (\nabla_{W} T)(X_1, X_2) \rangle_g = \frac{1}{2} \langle Z, \circlearrowright R^\nabla(X_1, X_2) W \rangle_g = 0.
\end{align*}
If all vector fields are horizontal, we can again use \eqref{CommuteBott} to verify that $\langle R(Z,W) X_1, X_2 \rangle_g = \langle R(X_1, X_2) Z, W \rangle_g$. In conclusion, $\Ric_{\ch}^* = \Ric_{\ch}$ and $C_{\Ric_\ch}(\alpha \wedge \beta) = (C_{\Ric_\ch}\alpha) \wedge \beta$.
\end{proof}

We are now in position to prove Proposition \ref{Cric}.

\begin{proof}
By the result in Lemma \ref{lemma:HorCurvOp}, it is sufficient to consider $\alpha \in \wedge^{(i,0)} T^*\M$. Hence, we can ignore the choice of $\ve$, since all such elements have the same length, independent of metric. The remaining proof follows the lines of \cite{GaMe73}. If $X_1, \dots, X_n$ is a local orthonormal basis of $\ch$, define an operator $\mathbf{R}: \wedge^\bullet T^*\M \to \wedge^\bullet T^*\M$
\begin{align*}
\mathbf{R} & = - \sum_{i,j,s=1}^n \flat X_j \wedge \iota_{X_i} \left( \flat X_s \wedge \iota_{R(X_i, X_j) X_i} \right) \\
& = - \sum_{i,j = 1}^n \flat X_j \wedge \iota_{R(X_i, X_j) X_i} - \sum_{i,j,s=1}^n   \flat X_j \wedge \flat X_s \wedge \iota_{R(X_i, X_j) X_s} \iota_{X_i} \\
& = C_{(\Ric_\ch)_{1,1}} + C_{(\Ric_\ch)_{2,2}} - \frac{1}{2} C_{\nu_0} = C_{\Ric_\ch} - \frac{1}{2} C_{\nu_0},
\end{align*}
where
$$\nu_0 =  \sum_{i,j,k=1}^n \flat X_j \wedge \flat X_k \otimes X_i \wedge  \left(\circlearrowright \nabla_{X_i } T(X_j , X_k) \right).$$ In the last expression, we have used that $- \sum_{i,j,s=1}^n \flat X_j \wedge \flat X_s \wedge \iota_{R(X_i, X_j) X_s} \iota_{X_i} = C_\nu$, where
\begin{align*}
\nu & =  - \sum_{i,j,s=1}^n \flat X_j \wedge \flat X_s \otimes X_i \wedge R(X_i, X_j) X_s \\
& =  - \frac{1}{2} \sum_{i,j,s=1}^n \flat X_j \wedge \flat X_s \otimes X_i \wedge ( R(X_i, X_j) X_s + R(X_j, X_s) X_i) \\
& \stackrel{\eqref{Bianchi}}{=}  \frac{1}{2} \sum_{i,j,s=1}^n \flat X_j \wedge \flat X_s \otimes X_i \wedge R(X_j, X_s) X_i  \\
& \qquad - \frac{1}{2} \sum_{i,j,s=1}^n \flat X_j \wedge \flat X_s \otimes X_i \wedge  \left(\circlearrowright \nabla_{X_i } T(X_j , X_s) \right) \\
& =  (\Ric_\ch)_{2,2} - \frac{1}{2} \nu_0.
\end{align*}

The operator $C_{\nu_0}$ vanishes on horizontal forms, permitting us to write
$$\langle C_{\Ric_\ch} \alpha, \alpha \rangle_\ve = \langle \mathbf{R} \alpha, \alpha \rangle.$$
Hence, it is sufficient to show that $\mathbf{R}$ is positive on $\wedge^{(i,0)} T^*M$ whenever $i \neq 0,n$.

Assume that the horizontal curvature operator is $R_\ch$ is positive, and let $S$ denote its square root. Then we have the following relation
\begin{align*}
{\bf R} & =  \sum_{i,j,r,s=1}^n \langle R(\flat X_i  \wedge \flat X_j), \flat X_r \wedge \flat X_s \rangle \flat X_j \wedge \iota_{X_i} \left( \flat X_s \wedge \iota_{ X_r } \right) \\
& = - \frac{1}{2} \sum_{a,b,j,s =1}^n  \flat X_j \wedge \iota_{\sharp \iota_{v_j} S(\flat X_a \wedge \flat X_b)} \left( \flat X_s \wedge \iota_{\sharp \iota_{X_s} S(\flat X_a \wedge \flat X_b)} \right) = - \frac{1}{2} \sum_{i,j=1}^n \mathbf{S}_{ij}^2 = - \sum_{i < j} \mathbf{S}_{ij}^2
\end{align*}
where $\mathbf{S}_{ij} = \sum_{k=1}^n  \flat X_k \wedge \iota_{\sharp \iota_{X_k} S(\flat X_i \wedge \flat v_j)}$ . Note that ${\bf S}_{ij}^* = - \mathbf{S}_{ij}$, so for any $\alpha \in \wedge^{(i,0)} T^*\M$,
$$\langle \mathbf{R}(\alpha), \alpha \rangle = \sum_{i <j} | \mathbf{S}_{ij}(\alpha) |^2.$$
Since $\mathbf{S}_{ij}(x) : \wedge^{(i,0)} T^*_x \M \to \wedge^{(i,0)} T^*_x \M$ in the basis $\flat X_1(x), \dots, \flat X_n(x)$ can be identified with the action of the matrix $(\langle \mathbf{S}(\flat X_i \wedge \flat X_s) , \flat X_r \wedge \flat X_s \rangle(x) )_{rs} \in \mathfrak{o}(n)$  acting on $\wedge^i \mathbb{R}^n$, $| \mathbf{S}_{ij}(\alpha)|^2$  is strictly positive for $i \neq 0, n$.

\end{proof}

\subsection{Ricci operator of $\hat{\nabla}^\ve$ under scaling} 
If we consider now  the operator $\Delta_{\ch, \ve}$ and its Ricci operator $\Ric^\ve$, determined by the curvature $\hat R^\ve$ of the connection $\hat \nabla^\ve$, then the identity $\hat \nabla^\ve_X Y = \nabla_X Y + \frac{1}{\ve} J_X Y$, gives us
\begin{align*}
\hat{R}^\eps(X, Y) &= R(X,Y) +  \frac{1}{\eps} \left((\nabla_X J)_Y - (\nabla_Y J)_X + J_{T(X,Y)} \right) +\frac{1}{\eps^2} [J_X, J_Y] \\
&= R(X,Y) +  \frac{1}{\eps} B_1(X,Y) +\frac{1}{\eps} B_2(X,Y) + \frac{1}{\ve^2} B_3(X,Y),
\end{align*} 
where $B_1(X,Y) = (\nabla_X J)_Y - (\nabla_Y J)_X $,  $B_2(X,Y) = J_{T(X,Y)}$ and $B_3(X,Y) = [J_X, J_Y]$. Furthermore, for $j =1,2, 3$, we consider $B_j$ as an element in $\Psi_{(2,2)}$ by formula
$$(\alpha \wedge \beta) B_j(v,w) = \langle B_j(v,w)\sharp^\ch \alpha, \sharp^\ch \beta\rangle_g,$$
and define $\check B_j$ in $\Psi_{(1,1)}$ by $\alpha \check B_j(v) = \tr_\ch \langle \hat B_j(\times, v) \sharp^\ch \alpha, \times \rangle_g$. Notice that $\check B_3 = 0$, $\check B_1$ is the dual of $\delta_\ch T$, while $\alpha( \check{B}_2(v)) = \langle \mathbf{J}^2 v, \sharp^\ch \alpha\rangle_g$.

With respect to the decomposition of the tangent space, we have that $C_{B_1}$, $C_{B_2}$ and $C_{B_3}$ maps $\wedge^{(i,j)} T^* \M$ into respectively $\wedge^{(i-1,j+1)} T^* \M \oplus \wedge^{(i-2, j+2)} T^* \M$, $\wedge^{(i,j)} T^* \M$ and $\wedge^{(i-2,j+2)} T^* \M$, and identical relations hold for $C_{\check{B}_j}$ for $j =1,2,3$.
We consider a $k$-form $\alpha$ as a sum $\alpha = \sum_{i+j=k} \alpha^{(i,j)}$, $\alpha^{(i,j)} \in \wedge^{(i,j)} T^*\M$ and use the convention that $\alpha^{(i,j)} = 0$ for $i$ and $j$ not in the permitted range $0 \leq i \leq n$ and $0 \leq j \leq m$. Then
\begin{align} \label{RicVeRicCh}
& \langle (C_{\Ric^\ve} - C_{\Ric_\ch}) \alpha, \alpha \rangle_\ve  \\ \nonumber
& = \frac{1}{\ve} \langle C_{B_2 + \hat B_2} \alpha, \alpha \rangle_\ve\ + \frac{1}{\ve} \sum_{i+j =k}  \langle C_{B_1 + \check{B}_1} \alpha^{(i+1,j-1)}, \alpha^{(i,j)} \rangle_\ve + \frac{1}{\ve^2} \sum_{i+j =k}  \langle C_{\ve B_1 + B_3} \alpha^{(i+2,j-2)}, \alpha^{(i,j)} \rangle_\ve
\end{align}
We remark also that relative to local orthonormal bases  $X_1, \dots, X_n$ and $Z_1, \dots, Z_m$, we have
\begin{align*}
& C_{\Ric_\ve} - C_{\Ric_\ch} \\
& = \frac{1}{\ve} \sum_{i=1}^m \flat \delta_\ch T(X_i) \wedge \iota_{X_i} + \frac{1}{\ve} \sum_{i,j,k=1}^n  \flat X_k \wedge \flat (\nabla_{X_k} T)(X_i, X_j) \wedge \iota_{X_j} \iota_{X_i} \\
& \quad + \frac{1}{\ve} \sum_{i,j=1}^n \sum_{r=1}^m \flat Z_r \wedge \flat (\nabla_{Z_r} T)(X_i, X_j) \wedge \iota_{X_j} \iota_{X_i}  + \frac{1}{\ve} \sum_{i=1}^n  \flat \mathbf{J}^2 X_i \wedge \iota_{X_i}\\
&  \quad + \frac{1}{2\ve} \sum_{i,j,k=1}^n \flat X_i \wedge \flat X_j \wedge \iota_{J_{T(X_j, X_i)} X_k} \iota_{X_k} + \frac{1}{\ve^2} \sum_{r,s=1}^m \sum_{k=1}^n \flat Z_r \wedge \flat Z_s \wedge \iota_{J_{Z_r} J_{Z_s} X_k} \iota_{X_k}.
\end{align*}

We have then the following lemma:

\begin{lemma}\label{lemma hj}
Assume that  $\nabla_Z T = 0$ for any $Z \in \Gamma^\infty(\V)$. There exist a constant $\ve >0$ and  a constant $c_2 \ge 0$ such that for any $\alpha \in \wedge^k T^* \M$  we have $|\langle (C_{\Ric^\ve} - C_{\Ric_\ch}) \alpha, \alpha \rangle_\ve | \leq \frac{c_2}{\sqrt{\ve}} \| \alpha \|_\ve^2$.
\end{lemma}

\begin{proof}
Write $\alpha = \sum_{i+j =k} \alpha^{(i,j)}$ according to the bi-grading. As we  are assuming that the torsion is parallel in vertical directions, we obtain that $\langle C_{B_1} \alpha^{(i+2,j-2)} , \alpha^{(i,j)} \rangle_g =0$. Hence, by \eqref{RicVeRicCh}, if we define $M_j$ such that $\langle C_{B_j + \check{B}_j} \alpha, \beta \rangle_g \leq M_j \|\alpha\|_g \|\beta \|_g$, then
\begin{align*} 
& \langle (C_{\Ric^\ve} - C_{\Ric_\ch}) \alpha, \alpha \rangle_\ve  \\
\leq &  M_2\sum_{i+j=k} \ve^{j-1}  \| \alpha^{(i,j)} \|^2_g  + M_1 \sum_{i+j =k}  \ve^{j-1} \| \alpha^{(i+1,j-1)} \|_g  \|\alpha^{(i,j)} \|_g   + M_3 \sum_{i+j =k} \ve^{j-2}  \|\alpha^{(i+2,j-2)}\|_g \| \alpha^{(i,j)} \|_g \\
\leq &   \frac{M_2}{\ve} \sum_{i+j=k}  \| \alpha^{(i,j)} \|^2_\ve  + \frac{M_1}{\sqrt{\ve}} \sum_{i+j =k} \| \alpha^{(i+1,j-1)} \|_\ve  \|\alpha^{(i,j)} \|_\ve  + \frac{M_3}{\ve} \sum_{i+j =k}  \|\alpha^{(i+2,j-2)}\|_\ve \| \alpha^{(i,j)} \|_\ve.
\end{align*}
Hence, there exists a constant $c_1$ such that for any $\alpha \in \wedge^k T^* \M$ and for sufficiently large $\ve$, we have $\langle (C_{\Ric^\ve} - C_{\Ric_\ch}) \alpha, \alpha \rangle_\ve | \leq \frac{c_1}{\sqrt{\ve}} \| \alpha \|_\ve^2$.
\end{proof}

\subsection{Semigroup of $\Delta_{\ch,\ve}$}
Similarly to the case for one-forms, we have a strongly continuous semigroup generated by $\Delta_{\ch,\ve}$.

\begin{lemma}\label{dve inter}
\

\begin{enumerate}[\rm (a)]
\item For $0 < \ve \leq \infty$, the operator $\Delta_{\ch,\ve}$ is hypoelliptic.
\item Let $0 < \ve < \infty$. The operator $\Dhe$  is the generator of a strongly continuous and smoothing semigroup of bounded operators on $L^2(\wedge^\bullet T^* \M, g_\ve)$. Moreover, if $\{ e^{t\Dhe}, t \ge 0 \}$ denotes this semigroup, then for every smooth one-form $\alpha$, $\alpha_t =e^{t\Dhe} \alpha$ is the unique solution of the heat equation:
\begin{align*}
\begin{cases}
\frac{\partial \alpha_t}{\partial t}= \Dhe \alpha_t , \\
\alpha_0=\alpha .
\end{cases}
\end{align*}
Furthermore, $d e^{t \Delta_\ch} = e^{t \Delta_\ch} d$.
\end{enumerate}
\end{lemma}

\begin{proof}
\
\begin{enumerate}[\rm (a)]
\item For $0 < \ve \le \infty$, we can write
\[
\Dhe= - (\nabla_\ch^\ve)^* \nabla_\ch^\ve -C_{\Ric^\ve}
\]
where $C_{\Ric^\ve}$ is a zero-order differential operator. Furthermore, we can locally write $ (\nabla_\ch^\ve)^* \nabla_\ch^\ve = - \sum_{i=1}^n \left(\nabla^\ve_{X_i} \right)^2-\nabla^\ve_{\nabla^\ve _{X_i} X_i}$ with $X_1,\cdots,X_n$ being a local horizontal frame that satisfies the two-step H\"ormander's condition. Again one can deduce hypoellipticity from arguments similar to Proposition 3.5.1 in \cite{Ponge}.
\item The argument is identical to the one given for one-forms in Proposition~\ref{Pol}.
\end{enumerate}

\end{proof}

\subsubsection{Proof of Theorem~\ref{th:Hk0}}

The proof of the main statement is now a direct consequence of the following lemma and otherwise identical to that of Theorem~\ref{H1 1forms}.

\begin{lemma}
Assume that $R_\ch >0$ and that $\nabla_Z T = 0$ for any $Z \in \Gamma^\infty(\V)$.  There exists a value $\ve >0$ such that:

\begin{enumerate}[\rm (a)]
\item   There exists a constant $c_3>0$ such that if $\alpha$ is a closed $k$-form with $m < k <n$,
$$\frac{1}{2} \Delta_{\ch,\ve} \| \alpha \|_\ve^2 - \langle \Delta_{\ch,\ve} \alpha , \alpha \rangle_\ve \geq c_3 \| \alpha\|_\ve^2.$$
\item  If $\alpha$ is a closed $k$-form with $m < k <n$, then in $L^2$, 
\[
\lim_{t \to +\infty} e^{t \Dhe} \alpha =0.
\]
\item If $\alpha$ is a closed $k$-form with $m < k <n$, then for every $t \ge 0$, $e^{t \Dhe} \alpha-\alpha$ is an exact smooth form.
\end{enumerate}
\end{lemma}
\begin{proof}

\

\begin{enumerate}[\rm (a)]
\item Notice that
\begin{align*}
& \frac{1}{2} \Delta_{\ch,\ve} \| \alpha \|_\ve^2 - \langle \Delta_{\ch,\ve} \alpha , \alpha \rangle_\ve = \| \nabla^\ve_\ch \alpha \|^2_\ve + \langle C_{\Ric^\ve} \alpha, \alpha \rangle_\ve \geq \langle C_{\Ric^\ve} \alpha, \alpha \rangle_\ve
\end{align*}

Since we assumed that the horizontal curvature operator was positive and since $\bigwedge^k T^* \M \subseteq \oplus_{i=1}^{n-1} \oplus_{j=0}^m \wedge^{(i,j)} T^*\M $, we have from Lemma ~\ref{lemma hj} and Proposition \ref{Cric} that there are constants $c_1>0$ and $c_2 \ge 0$ such that 
$$\langle C_{\Ric^\ve} \alpha, \alpha \rangle_\ve = \langle C_{\Ric_\ch} \alpha, \alpha \rangle_\ve + \langle C_{\Ric^\ve - \Ric_\ch} \alpha, \alpha\rangle_\ve \geq c_1 \| \alpha\|^2_\ve - \frac{c_2}{\sqrt{\ve}} \| \alpha \|^2_\ve .$$
The result follows.
\item The proof is identical to Lemma~\ref{lemma:Convergence}.
\item The proof is identical to Lemma~\ref{lemma:ClosedToClosed}.
\end{enumerate}
\end{proof}

\appendix
\section{Appendix: Geometric identities about adjoint connections}
In this Appendix we collect some formulas used in the text.

\subsection{Adjoint connections}
For any connection $\nabla$ on $T\M$ with torsion $T$, we define its adjoint connection $\hat{\nabla}$ by
$$\hat{\nabla}_X Y = \nabla_X Y - T(X,Y).$$
Notice that  the curvature of $\hat{\nabla}$ is given by
\begin{align} \nonumber
\hat{R}(X,Y)Z & =  R(X, Y)Z - (\nabla_XT)(Y, Z) - (\nabla_{Y} T)(Z,X) \\ \nonumber
& \qquad + T(X, T(Y,Z)) + T(Y,T(Z,X)) - T(T(X,Y), Z) \\ \nonumber
& =  R(X, Y)Z - d^\nabla T(X,Y,Z) + (\nabla_Z T)(X,Y) \\ \label{Curv1}
& = R(Z,Y)X - R(Z,X) Y + (\nabla_Z T)(X,Y).
\end{align}
Here we have used the first Bianchi identity 
\begin{equation} \label{Bianchi} \circlearrowright R(X,Y)Z = d^\nabla T(X,Y,Z) =  \circlearrowright T(T(X,Y) ,Z) + \circlearrowright (\nabla_X T)(Y,Z),\end{equation}
with $\circlearrowright$ denoting the cyclic sum. Notice in particular that
{\begin{equation} \label{Curv2} \hat{R}(X, Y)X = R(X, Y)X + (\nabla_X T)(X,Y).\end{equation}

\subsection{Connections compatible with the metric and commutation of the curvature}
Let $\nabla$ be any connection compatible with the metric $g$ and with torsion~$T$. Introduce a map $J$ by the formula
$\langle J_Z X, Y \rangle_g = \langle Z, T(X,Y) \rangle_g$. Then it is simple to verify that
$$\nabla_X Y = \nabla^g_X Y +\frac{1}{2} A(X,Y), \qquad A(X,Y) : = T(X,Y) - J_X Y - J_Y X.$$
Note that $\langle A(X,Y) , Z \rangle_g = - \langle Y, A(X,Z) \rangle_g$. Computations yield the following.
\begin{lemma}
For any connection $\nabla$, compatible with $g$, we have
\begin{align*}
& \langle R^\nabla(X,Y)Z, W \rangle_g - \langle R^\nabla(Z,W)X , Y \rangle_g \\
& = \frac{1}{2} \langle (\nabla_X A)(Y,Z), W \rangle_g - \frac{1}{2} \langle (\nabla_Y A)(X,Z), W \rangle_g \\
&\qquad - \frac{1}{2} \langle (\nabla_Z A)(W,X), Y \rangle_g + \frac{1}{2} \langle (\nabla_W A)(Z,X),Y\rangle_g \\
& \qquad + \frac{1}{2} \langle A(T(X,Y), Z) , W \rangle_g  - \frac{1}{2} \langle  A(T(Z,W), X) , Y\rangle_g \\
&\qquad + \frac{1}{4} \langle  A(Y,Z) , A(X,W) \rangle_g - \frac{1}{4} \langle A(Z,Y), A(W,X) \rangle_g\\
& \qquad - \frac{1}{4} \langle A(X,Z), A(Y,W) \rangle_g + \frac{1}{4} \langle A(Z,X), A(W,Y)\rangle_g 
\end{align*}
In particular,
\begin{enumerate}[\rm (a)]
\item If $T$ is skew-symmetric, i.e. $\langle T(X,Y), Z \rangle_g = - \langle Y, T(X,Z) \rangle_g$, then $A(X,Y) = T(X,Y)$ and
\begin{align} \label{Commute}
 &  \langle R^\nabla(X,Y)Z, W \rangle_g - \langle R^\nabla(Z,W)X , Y \rangle_g \\
= & \frac{1}{2} \langle (\nabla_X T)(Y,Z), W\rangle  - \frac{1}{2} \langle (\nabla_Y T)(X,Z), W \rangle_g  - \frac{1}{2} \langle (\nabla_Z T)(W,X) , Y\rangle + \frac{1}{2} \langle (\nabla_W T)(Z,X),Y\rangle_g .\notag
\end{align}
\item If $T(T(X,Y),Z) =0$, then
\begin{align} \label{CommuteBott}
 &  \langle R^\nabla(X,Y)Z, W \rangle_g - \langle R^\nabla(Z,W)X , Y \rangle_g  \\
=&  \frac{1}{2} \langle (\nabla_X A)(Y,Z), W \rangle_g - \frac{1}{2} \langle (\nabla_Y A)(X,Z), W \rangle_g  - \frac{1}{2} \langle (\nabla_Z A)(W,X) , Y \rangle_g  +  \frac{1}{2} \langle (\nabla_W A)(Z,X),Y\rangle_g  \notag
\end{align}
\end{enumerate}
\end{lemma}

\bibliographystyle{habbrv}
\bibliography{Bibliography}

\end{document}